\documentclass[reqno,a4paper,11pt]{amsart}

\usepackage[all,poly]{xy}
\usepackage{amsfonts}
\usepackage[mathcal]{eucal}
\usepackage{eufrak}
\usepackage{amssymb}
\usepackage{amsmath}
\usepackage{mathrsfs}
\usepackage{color}
\usepackage[colorlinks]{hyperref}
\usepackage{enumerate}

\theoremstyle{plain}
\newtheorem {lemma}{Lemma} 
\newtheorem {theorem}[lemma]{Theorem}

\newtheorem* {conjecture}{Conjecture}

\newtheorem {proposition}[lemma]{Proposition}

\theoremstyle{definition}
\newtheorem {remark}[lemma]{Remark}

\theoremstyle{definition}

{}

\newcommand{\M}{\operatorname{\mathbb M}}
\newcommand{\LL}{\operatorname{\mathcal L}}

\newcommand{\gr}{\operatorname{gr}}

\newcommand{\id}{\operatorname{id}}

\renewcommand{\Im}{\operatorname{Im}}

\title[The isomorphism conjectures for Leavitt path algebras]{A note on the isomorphism conjectures for\\ Leavitt path algebras}

\author{R.~Hazrat}\address{
School of Computing, Engineering and Mathematics\\
University of Western Sydney\\
Australia}

\email{r.hazrat@uws.edu.au}

\subjclass[2000]{16D70} 
\keywords{Path algebras, Leavitt path algebras, strongly graded rings}

\begin{document}
\begin{abstract}
We relate two conjectures which have been raised for classification of Leavitt path algebras. For purely infinite simple unital Leavitt path algebras, it is conjectured that $K_0$ classifies them completely~\cite{aalp,flowa}. For arbitrary Leavitt path algebras, it is conjectured that $K^{\gr}_0$ classifies them completely~\cite{hazann}. We show that for two finite graphs with no sinks (which their associated Leavitt path algebras include the purely infinite simple ones) if their $K^{\gr}_0$-groups of their Leavitt path algebras are isomorphic then their $K_0$-groups are isomorphic as well.  

\end{abstract}

\maketitle

 The theory of Leavitt path algebras were introduced in~\cite{aap05,amp} which associate to directed graphs certain type of algebras. These algebras were motivated by Leavitt's construction of universal non-IBN rings~\cite{vitt62}. One current direction in the theory is  focused on their classification. Motivated by the success of $K$-theory in classification of certain types of $C^*$-algebras~(\cite{phillips,rordam}) and the close connections of Leavitt path algebras with graph $C^*$-algebras, it is believed that a variant of $K$-groups ought to classify these objects. 
 
 We refer the reader to~\cite{aap05,aap06} for basics in the theory of Leavitt path algebras and~\cite{malaga} for their connections to graph $C^*$-algebras. For a graph $E$, the Leavitt path algebra in this note is over a fixed field $K$ and we simply denote the algebra by $\LL(E)$.

 Throughout this note, unless it is stated otherwise, all graphs are finite, all modules are right modules and $\mbox{gr-}A$ and $\mbox{mod-}A$ stand for the category of graded right $A$-modules and right $A$-modules, respectively. 

In the theory of Leavitt path algebras, two questions/conjectures have been put forward.

\begin{enumerate}[\upshape (i)]
\item \label{s11} Let $E$ and $F$ be finite graphs such that $\LL(E)$ and $\LL(F)$ are purely infinite simple. Then  
\begin{equation}\label{ookjh}
\big(K_0(\LL(E)),[\LL(E)]\big) \cong \big(K_0(\LL(F)),[\LL(F]\big)
\end{equation}
if and only if  $\LL(E) \cong \LL(F)$.

\item  \label{s22} Let $E$ and $F$ be any  graphs. Then there is an ordered preserving $\mathbb Z[x,x^{-1}]$-module  isomorphism 
\begin{equation}\label{ookjh2}
\big (K_0^{\gr}(\LL(E)),[\LL(E)]\big ) \cong \big (K_0^{\gr}(\LL(F)),[\LL(F)]\big )
\end{equation}
 if and only if $\LL(E)\cong_{\gr} \LL(F).$  
\end{enumerate}

The formula (1) means that there is an isomorphism between $K_0(\LL(E))$ and $K_0(\LL(F))$ such that $[\LL(E)]$ is sent to $[\LL(F)]$ by this isomorphism. 
Statement~(\ref{s11}) was raised as ``the classification question for purely infinite simple unital Leavitt path algebras" in~\cite{aalp}, where the authors provided affirmative answer for  the category of purely infinite simple  Leavitt path algebras whose graphs  have at most three vertices and no parallel edges. The question was further investigated in~\cite{flowa}, where among other things, it was proved that for two finite graphs $E$ and $F$, which their (purely infinite simple) Leavitt path algebras are Morita equivalent,  Statement~(\ref{s11}) is valid. 

Clearly Statement~(\ref{s11}) is not valid for all graphs. For example for the graphs 
\begin{equation*}
\xymatrix{
F: & \bullet \ar[r] & \bullet \ar[r] & \bullet, & \LL(F)\cong \M_3(K),\\
E:&  \bullet \ar[r]  & \bullet \ar@/^1.5pc/[r] & \bullet, \ar@/^1.5pc/[l] & \LL(E)\cong \M_3(K[x,x^{-1}]),\\
}
\end{equation*}
one can easily observe that although 
\[\Big(K_0(\LL(E)), [\LL(E)]\Big ) \stackrel{\cong}{\longrightarrow} \big(\mathbb Z, 3\big),\]
\[\Big(K_0(\LL(F)), [\LL(F)]\Big ) \stackrel{\cong}{\longrightarrow} \big(\mathbb Z, 3\big),\]
but $\LL(E)\not \cong \LL(F)$.

Statement~(\ref{s22}) was conjectured in~\cite{hazann} for all graphs, where the statement was proved for the class of polycephaly graphs (multi-headed graphs). In this note we refer to this statement as the \emph{graded Conjecture}.

In this paper we relate these two conjectures. We show that formula~(\ref{ookjh2}) implies~(\ref{ookjh}). 

Statement~(\ref{s22}) was raised in the graded setting. A Leavitt path algebra has a natural grading and it turned out that a purely infinite simple unital Leavitt path algebra with this grading is strongly graded. In fact in~\cite[Theorem~3.15]{haz} it was proved that for a finite graph $E$, its associated Leavitt path algebra $\LL(E)$ is strongly graded if and only if $E$ has no sinks. The proof given in~\cite{haz} is long and involves graph chasing. Here we first provide a short proof for this fact using the machinery 
of corner skew Laurent polynomial rings. This theorem will be used in a crucial way in our main Theorem~\ref{mainii}.

Let $R$ be a ring with identity and $p$ an idempotent of $R$. Let $\phi:R\rightarrow pRp$ be a corner isomorphism. A \emph{corner skew Laurent polynomial ring} with coefficients in $R$, denoted by $R[t_{+},t_{-},\phi]$, is a unital ring which is constructed as follows:  The elements of $R[t_{+},t_{-},\phi]$ are formal expressions
\[t^j_{-}r_{-j} +t^{j-1}_{-}r_{-j+1}+\dots+t_{-}r_{-1}+r_0 +r_1t_{+}+\dots +r_it^i_{+},\]
where $r_{-n} \in p_n R$ and $r_n \in R p_n$, for all $n\geq 0$, where $p_0 =1$ and $p_n =\phi^n(p_0)$. The addition is component-wise, and multiplication is determined by the distribution law and the following rules:
\begin{equation}\label{oiy53} 
t_{-}t_{+} =1, \qquad t_{+}t_{-} =p, \qquad rt_{-} =t_{-}\phi(r),\qquad  t_{+}r=\phi(r)t_{+}.
\end{equation}

The corner skew Laurent polynomial rings is studied in~\cite{arabrucom}, where its $K_1$-group is calculated. This construction is a special case of a so called fractional skew monoid rings constructed in~\cite{arafrac}. Assigning $-1$ to $t_{-}$ and $1$ to $t_{+}$ makes $A=R[t_{+},t_{-},\phi]$ a $\mathbb Z$-graded ring with $A=\bigoplus_{i\in \mathbb Z}A_i$, where $A_i=Rp_nt^i_{+}$, for $i>0$, $A_i=t^i_{-}p_nR$, for $i<0$ and $A_0=R$ (see~\cite[Proposition~1.6]{arafrac}). Clearly, when $p=1$ and $\phi$ is the identity map, then $R[t_{+},t_{-},\phi]$ reduces to the familiar ring $R[t,t^{-1}]$. 

Recall that an idempotent element $p$ of the ring $R$ is called a \emph{full idempotent} if $RpR=R$. 
\begin{proposition}\label{lanhc8}
Let $R$ be a ring with identity and $A=R[t_{+},t_{-},\phi]$ be a corner skew Laurent polynomial ring. Then $A$ is strongly graded if and only if $\phi(1)$ is a full idempotent. 
\end{proposition}
\begin{proof}
First note that $A_1=R\phi(1)t_{+}$ and $A_{-1}=t_{-}\phi(1)R$. Furthermore, assuming $\phi(1)=pcp$, $c\in R$, then 
\begin{equation*}
r_1\phi(1)t_{+} t_{-}\phi(1)r_2=r_1\phi(1)p\phi(1)r_2=r_1pcpppcpr_2=r_1\phi(1)\phi(1)r_2=r_1\phi(1)r_2.
\end{equation*}
Suppose $A$ is strongly graded. Then $1\in A_{1}A_{-1}$. That is 
\begin{equation}\label{ghgfri6}
1=\sum_i\big(r_i\phi(1)t_{+}\big)\big(t_{-}\phi(1)r'_i\big)=\sum_i r_i\phi(1)r'_i,
\end{equation}
where $r_i,r'_i\in R$. So $R\phi(1)R=R$, that is $\phi(1)$ is a full idempotent. 

On the other hand suppose $\phi(1)$ is a full idempotent. Since $\mathbb Z$ is generated by $1$, in order to prove that $A$ is strongly graded, it is enough to show that $1\in A_1A_{-1}$ and $1\in A_{-1}A_1$ (see~\cite[Proposition~1.1.1]{grrings}). But 
\[t_{-}\phi(1) \phi(1)t_{+}=t_{-1}\phi(1)t_{+}=1t_{-}t_{+}=1,\] shows that $1\in A_{-1}A_1$. Since $\phi(1)$ is a full idempotent, there are $r_i,r'_i\in R$, $i\in I$ such that 
$\sum r_i\phi(1)r'_i=1$. Then Equation~\ref{ghgfri6} shows that $1\in A_1A_{-1}$.  
\end{proof}

We can realise the Leavitt path algebras of finite graphs with no source in terms of corner skew Laurent polynomial rings. This way we can provide a short proof for
~\cite[Theorem~3.15]{haz}, which gives a criterion when a Leavitt path algebra associated to a finite graph is strongly graded.  

Let $E$ be a finite graph with no source and $E^0=\{v_1,\dots,v_n\}$ be the set of all vertices of $E$. For each $1\leq i \leq n$, we choose an edge $e_i$ such that $r(e_i)=v_i$ and consider $t_{+}=e_1+\dots+e_n \in \LL(E)_1$. Then $t_{-}= e^*_1+\dots+e^*_n$ is its right inverse.  
Thus by~\cite[Lemma~2.4]{arafrac}, $\LL(E)=\LL(E)_0[t_{+},t_{-},\phi]$, where $\phi:\LL(E)_0\rightarrow t_{+}t_{-}\LL(E)_0 t_{+}t_{-}$, $\phi(a)=t_{+}at_{-}$. 
Using this interpretation of Leavitt path algebras we are able to prove the following theorem.

\begin{theorem}\label{sthfin3}
Let $E$ be a finite graph.  Then $\LL(E)$ is strongly graded if and only if  $E$ does not have sinks.
\end{theorem}
\begin{proof}
We first prove the theorem for the case that $E$ is a finite graph with no sources. 
Write $\LL(E)=\LL(E)_0[t_{+},t_{-},\phi]$, where $\phi(1)=t_{+}t_{-}$. The theorem now follows from an easy to prove observation that $t_{+}t_{-}$ is a full idempotent if and only if $E$ does not have sinks along with Proposition~\ref{lanhc8}, that $\phi(1)$ is a full idempotent if and only if $\LL(E)_0[t_{+},t_{-},\phi]$ is strongly graded.

Now suppose $E$ is a finite graph containing source vertices. Let $u$ be a source in $E$.  Remove the vertex $u$ and all the edges emitting from it
from the graph $E$ and call this new graph $E'$. We prove that if $\LL(E')$ is strongly graded, then $\LL(E)$ is strongly graded. Note that there is natural graph morphism from $E'$ to $E$ which induces a graded monomorphism $\LL(E')\rightarrow \LL(E)$. So we identify $\LL(E')$ as a (non-unital) subring of $\LL(E)$. We need to show that for any $n\in \mathbb Z$, $1_{\LL(E)}\in \LL(E)_n\LL(E)_{-n}$. 
Write  \[1_{\LL(E)}=\sum_{v\in E^{0}} v=u+\sum_{v\in E'^{0}} v=u+1_{\LL(E')}.\] Since by assumption $\LL(E')$ is strongly graded, for any $n \in \mathbb Z$,  $1_{\LL(E')}\in \LL(E')_n\LL(E')_{-n} \subseteq \LL(E)_n\LL(E)_{-n}$. Thus we only need to show $u\in \LL(E)_n\LL(E)_{-n}$. 
Write $u=\sum_{\{\alpha \in E^1\mid s(\alpha)=u \}} \alpha \alpha^*$. Since for any such $\alpha$, $r(\alpha)\in E'^0$, and $\LL(E')$ is strongly graded, we have $r(\alpha)\in \LL(E')_0=\LL(E')_{n-1}\LL(E')_{-n+1}$. So 
\[\alpha\alpha^*=\alpha r(\alpha)\alpha^* \in \alpha \LL(E')_{n-1}\LL(E')_{-n+1} \alpha^* \subseteq   \alpha \LL(E)_{n-1}\LL(E)_{-n+1} \alpha^* \subseteq 
\LL(E)_n\LL(E)_{-n}. \] This shows that $u \in \LL(E)_n\LL(E)_{-n}$. 

Finally starting from the graph $E$, repeatedly removing the sources, since the graph is finite, in a finite number of steps we get a graph with no sources and no sinks. By the first part of the proof, the Leavitt path algebra of this graph is strongly graded. An easy induction now shows $E$ is strongly graded as well. 
\end{proof}

Let $A$ be a strongly $\mathbb Z$-graded ring (the following is in fact valid for any graded group $\Gamma$). By Dade's Theorem~\cite[Thm.~3.1.1]{grrings}, the functor $(-)_0:\mbox{gr-}A\rightarrow \mbox{mod-}A_0$, $M\mapsto M_0$, is an additive functor with an  inverse $-\otimes_{A_0} A: \mbox{mod-}A_0 \rightarrow \mbox{gr-}A$ so that it induces an equivalence of categories.  
 This implies
that $K_i^{\gr}(A)  \cong K_i(A_0)$, for $i\geq 0$.  Furthermore, $A_n \otimes_{A_0}A_m \cong A_{n+m}$ as $A_0$-bimodules  and $A_n\otimes_{A_0}A\cong A(n)$ as graded $A$-modules, where $n, m\in \mathbb Z$.

Let $E$ be a finite graph with no sinks. Set $A=\LL_K(E)$ which is a strongly $\mathbb Z$-graded ring by Theorem~\ref{sthfin3}.  For any $u \in E^0$ and $i \in \mathbb Z$, $uA(i)$ is a right graded finitely generated projective $A$-module and any graded finitely generated projective $A$-module is generated by these modules up to isomorphism, i.e., 
$$\mathcal V^{\gr}(A)=\Big \langle \big [uA(i)\big ]  \mid u \in E^0, i \in \mathbb Z \Big \rangle, 
$$ and $K_0^{\gr}(A)$ is the group completion of $\mathcal V^{\gr}(A)$. The action of $\mathbb N[x,x^{-1}]$ on $\mathcal V^{\gr}(A)$ and thus the action of $\mathbb Z[x,x^{-1}]$ on $K_0^{\gr}(A)$ is defined on generators by $x^j [uA(i)]=[uA(i+j)]$, where $i,j \in \mathbb Z$. We first observe that for $i\geq 0$, 

\begin{equation}\label{hterw}
x[uA(i)]=[uA(i+1)]=\sum_{\{\alpha \in E^1 \mid s(\alpha)=u\}}[r(\alpha)A(i)].
\end{equation}

First notice that for $i\geq 0$, $A_{i+1}=\sum_{\alpha \in E^1} \alpha A_i$. It follows \[uA_{i+1}=\bigoplus_{\{\alpha \in E^1 \mid s(\alpha)=u\}} \alpha A_i\] as $A_0$-modules. Using the fact that $A_n\otimes_{A_0}A\cong A(n)$, $n \in \mathbb Z$, and the fact that $\alpha A_i \cong r(\alpha) A_i$ as $A_0$-module,
we get \[uA(i+1) \cong \bigoplus_{\{\alpha \in E^1 \mid s(\alpha)=u\}} r(\alpha) A(i)\] as graded $A$-modules. This gives~(\ref{hterw}).

\begin{theorem}\label{mainii}
Let $E$ and $F$ be  finite graphs with no sinks such that 
\[\big (K_0^{\gr}(\LL(E)),[\LL(E)]\big ) \cong \big (K_0^{\gr}(\LL(F)),[\LL(F)]\big ).\] Then there is an isomorphism  
 $K_0(\LL(E)) \cong K_0(\LL(F)),$ where $[\LL(E)]$ is sent to $[\LL(F)].$ 
\end{theorem}
\begin{proof}
Set $A=\LL(E)$. Recall that the shift functor $T_1:\gr-A\rightarrow \gr-A$, $P\mapsto P(1)$ induces an isomorphism $K_0(T_1):K^{\gr}_0(A)\rightarrow K^{\gr}_0(A)$. Consider the map $\phi:=K_0(T_1)-\id$. We first show that the following sequence is exact. 
\begin{equation}\label{kaler}
K_0^{\gr}(A) \stackrel{\phi}{\longrightarrow}
K_0^{\gr}(A) \stackrel{U}\longrightarrow K_0(A) \longrightarrow 0.
\end{equation}
Here $U$ is the homomorphism induced by the forgetful functor.

Since the graded finitely generated projective modules of $A=\LL(E)$ are generated by the set $\{uA(i) \mid u\in E^0, i\in \mathbb Z\}$ and the finitely generated projective modules of $\LL(A)$ are generated by $\{uA \mid u\in E^0\}$, it follows that $U$ is an epimorphism.

Since each element of $K^{\gr}_0(A)$ is of the form $[P]-[Q]$, where $P$ and $Q$ are graded finitely generated $A$-module, we have
\[U\phi\big([P]-[Q]\big)=U\big([P(1)]-[Q(1)]-[P]+[Q]\big)=[P]-[Q]-[P]+[Q]=0.\] Thus $\Im(\phi) \subseteq \ker(U)$.

We will show that $\ker(U)\subseteq \Im(\phi)$.  We first show that for $u\in E^0$ and $i,j \in \mathbb Z$, we have 
\begin{equation}\label{inyouroom}
[uA(i)]-[uA(j)]\in \Im(\phi).
\end{equation}
Suppose $i> j$. (The case $j<i$ is similar.) Set $P=uA(j)\oplus uA(j+1)\oplus\dots\oplus uA(i-1)$. It is easy to see that $\phi(p)=[uA(i)]-[uA(j)]$. 

 Suppose $a=[P]-[Q] \in K^{\gr}_0(A)$ such that $U(a)=0$. Thus $[P]=[Q]$ in $K_0(A)$. So $P\oplus A^k\cong Q\oplus A^k$, for some $k\in \mathbb N$. Without changing $a$ we replace $P$ by $P\oplus A^k$ and $Q$ by $Q\oplus A^k$. (Note that here $A^k$ is considered as a graded $A$-module.) Thus $[P]=[Q]$ in  $\mathcal V(R)$. 

As a graded $A$-module 
\[P\cong u_1A(k_1)\oplus\dots\oplus u_nA(k_n),\] where $u_i\in E^0$ and $k_i\in \mathbb Z$. Similarly, 
$Q\cong v_1A(l_1)\oplus\dots\oplus v_mA(l_m)$, where $v_i\in E^0$ and $l_i\in \mathbb Z$.  

Consider the graded projective module $P'=u_1A\oplus\dots\oplus u_nA$ (which is obtained from $P$ by removing all the shifting). Then from~(\ref{inyouroom}) it follows that $[P]-[P'] \in \Im(\phi)$. Similarly if $Q'=v_1A\oplus\dots\oplus v_mA$ then $[Q]-[Q']\in \Im(\phi)$. We will show that $[P']-[Q']\in \Im(\phi)$, which then implies $a=[P]-[Q]\in \Im(\phi)$.  Note that since $U([P]-[Q])=0$, it follows $U([P']-[Q'])=0$. 

Recall that there is an isomorphism $\mathcal V(A) \cong M_E$, $[uA]\mapsto [u]$. Combining this with the epimorphism $U:\mathcal V^{\gr}(A) \rightarrow \mathcal V(A)$ and passing $[P']$ and $[Q']$ to $M_E$, since $[P']=[Q']$ in $\mathcal V(A)$, we get
\[[u_1]+\dots+ [u_n] = [v_1]+\dots+[v_m].\]
Set $p'=[u_1]+\dots+ [u_n]$ and $q'=[v_1]+\dots+[v_m]$. By~\cite[Lemma~4.3]{amp}, $p'=q'$ in $M_E$ if and only if $p'\rightarrow s$ and $q'\rightarrow s$, where $s=[w_1]+\dots+[w_t]  \in F$. 
Choose  the graded finitely generated projective $A$-module $S=w_1A\oplus \dots \oplus w_tA$ (i.e., with no shifting) as a preimage for $s$. Then $[S]\in \mathcal V^{\gr}(A)$ maps to $s\in M_E$. We will show that $[P']-[S]\in \Im(\phi)$ and  $[Q']-[S]\in \Im(\phi)$. It then follows $[P']-[Q']\in \Im(\phi)$. We proceed by induction. Suppose $p'\rightarrow_1 s$. Then by the definition, for a vertex, say $u_1$ in $p'$, \[s=\sum_{\{\alpha\in E^1\mid s(\alpha)=u_1\}}[r(\alpha)]+[u_2]+\dots+ [u_n].\] Then 
\begin{equation*}
[S]=\sum_{\{\alpha\in E^1\mid s(\alpha)=u_1\}}[r(\alpha)A]\oplus [u_2A\oplus\dots\oplus u_nA]
= [u_1A(1)\oplus u_2A\oplus\dots\oplus u_nA],
\end{equation*}
as $[u_1A(1)]=\sum_{\{\alpha\in E^1\mid s(\alpha)=u_1\}}[r(\alpha)A]$ (see~(\ref{hterw})). Since $[P']= [u_1A\oplus u_2A\oplus\dots\oplus u_nA]$, from~(\ref{inyouroom}) it follows that 
$[P']-[S]\in \Im(\phi)$. Now an easy induction shows if $p'\rightarrow s$ then $[P']-[S]\in \Im(\phi)$. Similarly $[Q']-[S]\in \Im(\phi)$. Thus we proved that the sequence~(\ref{kaler}) is exact. 

Suppose $\psi: K^{\gr}_0(\LL(E))\rightarrow K^{\gr}_0(\LL(F))$ is an ordered preserving $\mathbb Z[x,x^{-1}]$-isomorphism such that $\psi([\LL(E)])=[\LL(F)]$. 
For any generator $[P]$ of $K^{\gr}_0(\LL(E))$ we have 
\[\psi\phi([P])=\psi([P(1)]-[P])=\psi(x[P])-\psi([P])=x\psi([P])-\psi([P])=\psi([P])(1)-\psi([P])=\phi\psi([P]).\]
This shows that the left hand side of the following diagram is commutative 
\begin{equation}\label{lol153332}
\xymatrix@=15pt{
K^{\gr}_0(\LL(E)) \ar[d]^{\psi} \ar[r]^{\phi} &
K^{\gr}_0(\LL(E)) \ar[d]^{\psi} \ar[r]& K_0(\LL(E))\ar@{.>}[d] \ar@{.>}[d]\ar[r] &0\\
K^{\gr}_0(\LL(F)) \ar[r]^{\phi}  &
K^{\gr}_0(\LL(F)) \ar[r]& K_0(\LL(F)) \ar[r] \ar[r] &0,}
\end{equation}
and so induces a natural order preserving isomorphism, call it $\psi$ again, $\psi:K_0(\LL(E)\rightarrow K_0(\LL(F))$ such that $\psi([\LL(E)])=[\LL(F)].$ This completes the proof. 
\end{proof}

\begin{remark}
\begin{enumerate}

\item The converse of the Theorem~\ref{mainii} is not valid. Namely if for two finite graphs $E$ and $F$, 
$\big(K_0(\LL(E)),[\LL(E)]\big) \cong \big(K_0(\LL(F)),[\LL(F]\big)$, this does not imply that $K^{\gr}_0(\LL(E))\cong K^{\gr}_0(\LL(E))$. For example for graphs 
\medskip
\begin{equation*}
{\def\labelstyle{\displaystyle}
E : \quad \,\, \xymatrix{
 \bullet  \ar@(lu,ld)\ar@/^0.9pc/[r] & \bullet \ar@/^0.9pc/[l] 
}} \qquad  \quad
{\def\labelstyle{\displaystyle}
F: \quad \,\, \xymatrix{
 \bullet  \ar@(lu,ld)\ar@/^0.9pc/[r] & \bullet \ar@(ru,rd)\ar@/^0.9pc/[l] 
}} \qquad  \quad 
\end{equation*}

\medskip
one can calculate that $K_0(\LL(E))=K_0(\LL(F))=0$ but $K_0^{\gr}(\LL(E)) \cong \mathbb Z\oplus \mathbb Z$, and $K_0^{\gr}(\LL(F))=\mathbb Z[1/2]$ 
(see~\cite[Example~11]{hazann}). 

\item The isomorphism of formula~(\ref{ookjh2})
\begin{equation*}
\big (K_0^{\gr}(\LL(E)),[\LL(E)]\big ) \cong \big (K_0^{\gr}(\LL(F)),[\LL(F)]\big )
\end{equation*}
implies that 
\begin{equation}\label{kgrd83}
K_i^{\gr}(\LL(E)) \cong  K_i^{\gr}(\LL(F)), i\geq 0.
\end{equation}
Indeed, since $\LL(E)$ and $\LL(F)$ are strongly graded, applying the Dade's Theorem~\cite[Thm.~3.1.1]{grrings}, we get 
\begin{equation*}
\big (K_0(\LL(E)_0),[\LL(E)_0]\big ) \cong \big (K_0(\LL(F)_0),[\LL(F)_0]\big ).
\end{equation*} 
Since $\LL(E)_0$ and $\LL(F)_0$ are ultramatricial algebras,  by \cite[Theorem~15.26]{goodearlbook} $\LL(E)_0\cong \LL(F)_0$. Again using Dade's Theorem we get 
\[ K^{\gr}_i(\LL(E))\cong K_i(\LL(E)_0)\cong K_i(\LL(F)_0)\cong  K^{\gr}_i(\LL(F)).\]

\end{enumerate}
\end{remark}

\begin{remark}

We relate the exact sequence~\ref{lol153332} with the Abrams-Tomforde conjecture. 

\begin{conjecture}{\sc (The Abrams-Tomforde Conjecture)} If $E$ and $F$ are graphs, then $\LL_{\mathbb C}(E) \cong \LL_{\mathbb C}(F)$ (as rings) implies that 
$C^*(E) \cong C^*(F)$ (as $*$-algebras).
\end{conjecture}

The conjecture was raised in~\cite{genetomforde}, where an affirmative answer was given for the case that the graphs are row-finite, cofinal, satisfy Condition (L), and contain at least one cycle. These graphs have no sinks and the graph $C^*$-algebra of such graphs are purely infinite and simple. Thus the Kirchberg-Phillips classification theorem can be applied~\cite{phillips}. Starting from $\LL_{\mathbb C}(E) \cong_{\gr} \LL_{\mathbb C}(F)$, (or just the formula~(\ref{ookjh2})), the following diagram shows how we can systematically obtain that $K_0(C^*(E))\cong K_0(C^*(F))$ and $K_1(C^*(E))\cong K_1(C^*(F))$, and so by Kirchberg-Phillips Theorem, $C^*(E) \cong C^*(F)$, when $E$ and $F$ are finite. The diagram is the combination of~(\ref{lol153332}) and observing that the map $\phi$ coincides with $1-N^t$,
 along with the exact sequence obtained by Raeburn and Szyma\'nski~\cite[Theorem~3.2]{Raeburn113}  (see~\cite[Remark~5]{hazann} for details). 

\begin{equation*}
\xymatrix@=15pt{0 \ar[r] & K_1(C^*(E)) \ar[r] \ar@{.>}[ddd] &
K_0\big(C^*(E\times_1 \mathbb Z)\big) \ar[d]^{\cong}
\ar[r]^{1-\beta^{-1}_*} & K_0\big(C^*(E\times_1 \mathbb Z)\big) \ar[d]^{\cong} \ar[r]& K_0(C^*(E))\ar@{.>}[d]\ar[r]& 0\\
& & K^{\gr}_0(\LL_{\mathbb C}(E)) \ar[d]^{\cong} \ar[r]^{1-N^{t}} &
K^{\gr}_0(\LL_{\mathbb C}(E)) \ar[d]^{\cong} \ar[r]& K_0(\LL_{\mathbb C}(E))\ar@{.>}[d] \ar@{.>}[d]\ar[r] &0\\
& & K^{\gr}_0(\LL_{\mathbb C}(F)) \ar[r]^{1-N^{t}} \ar[d]^{\cong} &
K^{\gr}_0(\LL_{\mathbb C}(F)) \ar[r]\ar[d]^{\cong}& K_0(\LL_{\mathbb C}(F)) \ar[r] \ar@{.>}[d]\ar[r] &0\\
0 \ar[r] & K_1(C^*(F)) \ar[r]  &
K_0\big(C^*(F\times_1 \mathbb Z)\big) 
\ar[r]^{1-\beta^{-1}_*} & K_0\big(C^*(F\times_1 \mathbb Z)\big)  \ar[r]& K_0(C^*(F)) \ar[r]& 0.}
\end{equation*}
\end{remark}

\end{document}